\documentclass[12pt,psfig,reqno]{amsart}
\usepackage{amsmath}
\usepackage{bbm}
\usepackage{}
\usepackage{amssymb}
\usepackage{amsfonts}
\usepackage{amscd}
\usepackage{amssymb,amsfonts,latexsym}
\usepackage{graphics,verbatim}
\usepackage{graphicx}
\usepackage[usenames]{color} 
\makeatletter

\renewcommand\section{\@startsection {section}{1}{\z@}%
                                   {-3.5ex \@plus -1ex \@minus -.2ex}%
                                   {2.3ex \@plus.2ex}%
                                   {\centering\normalfont\bf}}

 \numberwithin{equation}{section}
\numberwithin{equation}{section}

\numberwithin{equation}{section}

\theoremstyle{plain}
\newtheorem{thm}{Theorem}[section]
\newtheorem{lemma}[thm]{Lemma}
\newtheorem{pro}[thm]{Proposition}

\newtheorem*{thm*}{Theorem}

\begin{document}
\title{Connectedness of self-affine sets with product digit sets}
\author{Jing-Cheng Liu$^{1}$, Jun Jason Luo$^{2}$ and Ke Tang$^{1}$}
\address{1 Key Laboratory of High Performance Computing and Stochastic Information Processing (Ministry of Education of China), College of Mathematics and
Computer Science, Hunan Normal University, Changsha, Hunan 410081, China} \email{liujingcheng11@126.com}
\address{2 College of Mathematics and Statistics, Chongqing University, Chongqing 401331, China}\email{jun.luo@cqu.edu.cn}

\date{\today}
\keywords {Connectedness, self-affine set, self-affine tile, tile digit set.}
\subjclass[2010]{Primary 28A80; Secondary 52C20, 52C45.}
\thanks{The research is supported in part by the NNSF of China (No.11571104, No.11301175, No.11301322), the Hunan Provincial NSF (No.13JJ4042),
the program for excellent talents in Hunan Normal University (No.ET14101), Specialized Research Fund for the Doctoral Program of Higher Education of China (20134402120007), Foundation for Distinguished Young Talents in Higher Education of Guangdong (2013LYM\_0028).}

\begin{abstract}
Let $T(A,\mathcal{D})$ be a self-affine set generated by an expanding matrix $A=\begin{pmatrix}
 p& 0 \\
 -a&q
\end{pmatrix}$ and a product digit set $\mathcal{D}=\{0,1,\dots,m-1\}\times \{0,1,\dots,n-1\}$. We provide a necessary and sufficient condition for the $T(A,\mathcal{D})$ to be connected, which generalizes the known results.
\end{abstract}

\maketitle

\section{\bf Introduction}
Let $A$ be an $n\times n$ expanding matrix (i.e., all its eigenvalues have moduli $> 1$), and let ${\mathcal D}\subset {\mathbb R}^n$ be a finite set. It is well known  that there is a unique nonempty compact set $T:=T(A,\mathcal{D})$ \cite{H81} satisfying the set-valued functional equation
\begin{equation}\label {eq 1.2}
T=\bigcup_{d\in {\mathcal D}}A^{-1}(T+d).
\end{equation}
More explicitly, $T$ can be given by the radix expansions:
\begin{equation}\label {eq 1.3}
T=\left\{\sum_{i\ge 1} A^{-i}d_{j_i}: d_{j_i}\in \mathcal{D}\right\}.
\end{equation}

We call $\mathcal D$ a {\it digit set} and $(A, {\mathcal D})$ an {\it affine pair}. Then $T$ is usually called a {\it self-affine set} generated by $(A, {\mathcal D})$. Furthermore, if $\#{\mathcal D}=|\det(A)|$ is an integer and $T$ has positive Lebesgue measure, then $T$ can tile ${\mathbb R}^n$ in the sense: there exists a discrete set $\mathcal{J}\subset \mathbb{R}^n$ such that $T+\mathcal {J}=\mathbb{R}^n$ and $(T^{\circ} +t)\cap(T^{\circ} +s)=\emptyset$ for all distinct $t, s \in \mathcal {J}$. We call such $T$ a {\it self-affine tile} and ${\mathcal D}$ a {\it tile digit set}.

The fundamental theory of self-affine tiles was established by Kenyon \cite{Ke92} and Lagarias and Wang \cite{LW1,LW2,LW3} in the last 90's. Since then, It has been a hot topic to study self-affine tiles and related fields. For example, people studied the fractal structure of their boundaries (\cite{LLu3},\cite{SW99},\cite{V98}),  their dynamical properties \cite{So97} and their applications to wavelets (\cite{BJ99},\cite{St93},\cite{W02}). Especially, one of the interesting aspects is the topological properties of self-affine tiles, such as connectedness (\cite{AG04},\cite{AT04},\cite{GH94},\cite{HSV94},\cite{KL00},\cite{KLR04}); disk-likeness (\cite{BG94},\cite{BW01},\cite{DL11},\cite{K10},\cite{LL07},\cite{LuLu}); local connectedness \cite{Lu07}. On the other hand, the self-affine sets arise in fractal geometry, as a class of important  fractal sets, to investigate  their topological structure  is also a basic and central theme (\cite{AG04},\cite{H85},\cite{LuLR13},\cite{LLu1},\cite{LLu2},\cite{LLX15},\cite{LRT02},\cite{MDD14},\cite{NT1},\cite{NT2}).

Among the above studies of the topology of self-affine tiles $T(A, {\mathcal D})$, people mainly focus on  a kind of consecutive collinear digit sets, that is similar to the one-dimensional case, i.e., ${\mathcal D}=\{0,1,\dots, \ell\}v$ where $v\in {\mathbb R}^n\setminus\{0\}$ (see \cite{AG04},\cite{AL10},\cite{KL00},\cite{KLR04},\cite{LL07}). In an attempt to consider the tiles generated by more general digit sets, recently,
 Deng and Lau \cite{DL11} tried a class of simple tiles $T(A, {\mathcal D})$ where $A$ is a lower triangular and $\mathcal D$ is arranged into a rectangular form. They proved an interesting result:

\begin{thm} [\cite{DL11}]\label{DL-thm}
 Let $p,q\in\mathbb{Z}$ with  $|p|, |q|\geq 2$ and $a\in \mathbb{R}$. Let
  \begin{equation*}
 A=\begin{pmatrix}
 p& 0 \\
 -a&q
\end{pmatrix}, \quad  \mathcal {D}=\left\{\dbinom{i}{j}: 0\leq i\leq |p|-1, 0\leq j\leq |q|-1\right\}.
\end{equation*}
 Then $T$ is a self-affine tile, and it is connected if and only if $|a|\le |{q(q-{\rm{sgn}}(p))}|$,  where $\rm{sgn}(p)$ denotes the sign of $p$.
\end{thm}

By exchanging the column and row of the above digit set $\mathcal D$, a sufficient condition was also obtained by Ma {\it et al.}:

\begin{thm}[\cite{MDD14}] \label{MDD-thm}
Let $p,q\in\mathbb{Z}$ with  $3\leq |p|+1<|q|<2|p|-1$ and $a\in\mathbb{R}$. Let
\begin{equation*}
 A=\begin{pmatrix}
 p& 0 \\
 -a&q
\end{pmatrix}, \quad \mathcal{D}=\left\{\dbinom{i}{j}: 0\leq i\leq |q|-1, 0\leq j\leq |p|-1\right\}.
\end{equation*}
If $q^2-|pq|\leq |a| \leq\frac{q^2(|p|-1)}{|q|-2}$, then $T$ (may be not a tile) is connected.
\end{thm}

In this paper, we generalize the above two results to self-affine sets  and obtain the following necessary and sufficient condition for connectedness.

\begin{thm}\label{th(1.1)}
 Let $p,q,\in\mathbb{Z}$ with  $|p|, |q|\geq 2$,  $a\in \mathbb{R}$, and let  $|p|+1<m<2|p|-1$,
$n\geq (|q|+1)/2$. Let
  \begin{equation*}
 A=\begin{pmatrix}
 p& 0 \\
 -a&q
\end{pmatrix},\quad  \mathcal{D}=\left\{\dbinom{i}{j}: 0\leq i\leq m-1, 0\leq j\leq n-1 \right\}.
\end{equation*}

(i) If $|q|=2$, then $T(A,\mathcal{D})$ is connected for any $a$.

(ii) If $|q|\ge 3$, then $T(A,\mathcal{D})$ is connected if and only if
$$|q|(|q|-n)\leq |a| \leq \frac{q^2(n-1)}{|q|-2} \quad \text{or} \quad |q|-n\leq |a| \leq \frac{|q|(n-1)}{|q|-2}.$$
\end{thm}

We notice that: when $n=|q|$, the above $T(A,\mathcal{D})$ is connected if and only if $|a|\le \frac{q^2(|q|-1)}{|q|-2}$. That is an  extension of Theorem \ref{DL-thm}; when $n=|p|$, the case (ii) of the theorem is a stronger version of Theorem \ref{MDD-thm}.

In the theory of self-affine tiles,  to characterize the tile digit sets $\mathcal{D}$ for a given expanding matrix is a very challenging problem (even in ${\mathbb R}^1$, see \cite{LLR13}). At the last section of remarks in \cite{DL11}, the authors doubt whether the above $T(A,\mathcal{D})$ is a tile or not, provided $mn=|pq|$. Here we can give some negative answers.

\begin{pro}\label{th(1.2)}
Let $|p|, |q|\geq 2, m, n>0$ be integers with $|pq|=mn$, and $a\in \mathbb{R}$.  Let
\begin{equation*}
 A=\begin{pmatrix}
 p& 0 \\
 a&q
\end{pmatrix}  \quad \text{and}\quad \mathcal{D}=\left\{\dbinom{i}{j}: 0\leq i\leq m-1, 0\leq j\leq n-1 \right\}.
\end{equation*} Then we have

(i)   $m<|p|$,  $T(A,\mathcal {D})$ is not a tile.

(ii)  $m=|p|$, $T(A,\mathcal {D})$ is  a tile.

(iii)  $m> |p|$,

   \ \ \ (a) If $a=0$, then  $T(A,\mathcal {D})$ is not  a tile;

  \ \ \  (b) If $a\in {\mathbb Z}\setminus \{0\}$  and $2n-1\geq |q|$,  then $T(A,\mathcal {D})$ is not a tile.
\end{pro}

For the organization of the paper, we give some basic notation and tools in Section 2. We prove Theorem \ref{th(1.1)}  by four key lemmas in Section 3, and prove Proposition \ref{th(1.2)} in Section 4.

\section{\bf Preliminaries \label{sect.2}}

In this section, we provide some basic notation and results that play an important role in the proofs of the paper. Given two positive integers $m$ and $n$, we write $E_{m}=\{0,1,\dots, m-1\}$, $E_{n}=\{0,1,\dots, n-1\}$, and define the difference sets $\Delta E_m=E_m-E_m, \Delta E_n= E_n-E_n$. For the expanding matrix and the product digit set as in the last section
$$
A=\begin{pmatrix}
 p & 0 \\
 -a & q
\end{pmatrix}, \quad  \mathcal{D}=E_m\times E_n.$$
By simple calculations, we have
\begin{equation*}
A^{-k}=\begin{pmatrix}
 p^{-k} &0\\
r_ka&q^{-k}
\end{pmatrix},\quad k\geq 1,
\end{equation*}
where
\begin{equation}\label{eq2.3}
r_k:=\left\{
\begin{array}{ll}
(p^{-k}-q^{-k})/(q-p) & p\neq q \\
k/q^{k+1}& p=q.
\end{array}
\right.
\end{equation}

Following the notation of  \cite{DL11}, we denote two sets of sequences by
$${\mathcal I}_1=\{{\bf i}=i_1i_2\cdots: i_k\in{E}_{m}\} \quad\text{and}\quad {\mathcal I}_2=\{{\bf j}=j_1j_2\cdots: j_k\in{E}_{n}\}.$$
Then the self-affine set generated by $A$ and $\mathcal D$ can be written in the following explicit form:
\begin{equation}\label{eq2.4}
    T(A,\mathcal{D})=\Bigg\{
       \begin{pmatrix}
          p({\bf i})\\
            ar({\bf i})+ q({\bf j})
       \end{pmatrix}
              :{\bf i}\in {\mathcal I}_1, {\bf j}\in{\mathcal I}_2\Bigg\}.
\end{equation}
where
$$p({\bf i})=\sum_{k\ge 1}i_kp^{-k},\quad r({\bf i})=\sum_{k\ge 1}r_k i_k,\quad q({\bf j})=\sum_{k\ge 1}j_kq^{-k}.$$

Obviously, it is not hard to verify that $T(A,\mathcal{D})$ is connected if and only if $T(-A,\mathcal{D})$ is connected (see \cite{KL00} or \cite{MDD14}). So we may assume $p>0$ in the whole paper. In general, there is a fundamental criterion  on the connectedness of fractal sets (see, for example \cite{H85} or \cite{KL00}):

\begin{lemma}\label{th(2.1)}
 Let $\{\psi_j\}_{j=1}^N$ be a family of contractions on $\mathbb{R}^n$ and let $T$ be its attractor. Then $T$ is connected if and only if for any
$i\neq j \in  \{1,\dots,N\}$, there exists a sequence of indices $j_1,\dots,j_\ell$ in $\{1, \dots, N\}$, with $j_1=i$ and $j_\ell=j$, such that $\psi_{j_k}(T)\cap \psi_{j_{k+1}}(T) \neq \emptyset$ for all $1\leq k\leq \ell-1$.
\end{lemma}

However, in order to obtain the main theorems of the paper, we need to define the following useful tools:

\begin{align*}
&\mathcal{A}=  \{(a_1,a_2,\dots) : \ \sum_{k\ge 1}a_k p^{-k}=1, a_k \in \Delta E_m\},\\
&\mathcal{B}= \{(b_1,b_2,\dots) : \ \sum_{k \ge 1}  b_k p^{-k}=0, b_k \in \Delta E_m\},\\
&\mathcal{S}= \{\sum_{k \ge 1}a_k q^{-k}: \ (a_1,a_2,\dots)\in \mathcal{A}\}, \quad \mathcal{Q}= \{\sum_{k \ge 1}b_k q^{-k}: \ (b_1,b_2,\dots)\in \mathcal{B}\},\\
&\mathcal{S}'= \{\sum_{k \ge 1}ka_k p^{-k}: \ (a_1,a_2,\dots)\in \mathcal{A}\},\quad \mathcal{Q}'= \{\sum_{k \ge 1}kb_k p^{-k}: \ (b_1,b_2,\dots)\in \mathcal{B}\}.
\end{align*}

In the last part of this section, our main task is to calculate the maximums and minimums of ${\mathcal S}$, ${\mathcal S}'$ and ${\mathcal Q}$, ${\mathcal Q}'$ which will be used frequently in the next section.

As in \cite{MDD14}, if $p+1<m<2p-1$, we can divide ${\mathcal A}$  and $\mathcal{B}$ into  parts:
$$\mathcal{A}=\bigcup_{i=1}^3{\mathcal A}_i  \  \  \text{and}  \  \  \mathcal{B}=\bigcup_{i=1}^3\mathcal{B}_i$$
where
\begin{align*}
&\mathcal{A}_1=\{(p-1,c_1,c_2,\dots) : \sum_{k \ge 1} c_k p^{-k}=1,  c_k \in \Delta E_m\};\\
&\mathcal{A}_2=\{(p,c_1,c_2,\dots): \sum_{k \ge 1} c_k p^{-k}=0,  c_k \in \Delta E_m\};\\
&\mathcal{A}_3=\{(p+1,-c_1,-c_2,\dots) : \sum_{k \ge 1} c_k p^{-k}=1,  c_k \in \Delta E_m\};\\
&\mathcal{B}_1=\{(-1,c_1,c_2,\dots) : \sum_{k \ge 1} c_k p^{-k}=1, c_k \in \Delta E_m\};\\
&\mathcal{B}_2=\{(1,-c_1,-c_2,\dots) : \sum_{k \ge 1} c_k p^{-k}=1,  c_k \in \Delta E_m\};\\
&\mathcal{B}_3=\{(0,c_1,c_2,\dots): \sum_{ k \ge 1} c_k p^{-k}=0, c_k \in \Delta E_m\}.
\end{align*}

\begin{pro}\label{th(2.2)}
 Let $p, q, m \in \mathbb{Z}$, $p, |q|\geq 2$ and $p+1<m<2p-1 $, let $M_1=\max \{x:x\in \mathcal{S}\}$, $m_1=\min\{x:x\in \mathcal{S}\}$, then \\
(i) If $q>0$
\begin{equation}\label {eq2.5}
\left\{
\begin{array}{ll}
M_1=\frac {p-1}{q-1}, \quad    m_1=\frac {pq+q-2p}{q(q-1)}\qquad  p\ge q \\
\\
M_1=\frac {pq+q-2p}{q(q-1)}, \quad   m_1=\frac {p-1}{q-1} \qquad  p<q.
\end{array}
\right.
\end{equation}
(ii) If $q<0$
\begin{equation}\label {eq2.6}
M_1=\frac {pq+2p-q}{q(q+1)}, \quad   m_1=\frac {p+1}{q+1}.
\end{equation}
\end{pro}

\begin{proof}

Let $\mathcal{S}_i= \left \{\sum_{k \ge 1}a_k q^{-k}: (a_1,a_2,\dots) \in \mathcal{A}_i \right\}, i=1,2,3.$  Then $\mathcal{S}=\bigcup_{i=1}^3\mathcal{S}_i$.  The element of $\mathcal{S}_i$, say $x_i$, can be written explicitly as follows:
\begin{align}\label {eq2.7}
x_1=\frac {p}{q}+\frac {s-1}{q}, \ \ \ \ \ \  x_2=\frac {p}{q}+\frac {t}{q},  \ \ \ \ \ \  x_3=\frac {p}{q}+\frac {1-s}{q},
\end{align}  where $s\in \mathcal{S}, t\in \mathcal{Q}$.

Obviously, the four sequences $(p-1,p-1,p-1,\dots), (p+1,-p+1,-p+1,\dots), (p+1,-p-1,p+1,-p-1,\dots), (p-1,p+1,-p-1,p+1,\dots)$ belong to $\mathcal{A}$. Hence
\begin{equation*}
 \frac {p-1}{q-1},\quad \frac {pq+q-2p}{q(q-1)},\quad \frac {p+1}{q+1},\quad \frac {pq+2p-q}{q(q+1)}\in\mathcal{S}.
\end{equation*}

We first show that the  maximum and minimum of $\mathcal{S}$ can not be taken from $\mathcal{S}_2$. As $\mathcal{B}=\mathcal{B}_1\cup\mathcal{B}_2\cup\mathcal{B}_3$, we have $\mathcal{S}_2=\mathcal{S}^1_2\cup\mathcal{S}^2_2\cup\mathcal{S}^3_2$ where
$$\mathcal{S}^j_2=\frac{p}{q}+\frac{1}{q}\left \{\sum_{k \ge 1}c_k q^{-k}: (c_1,c_2,\dots) \in \mathcal{B}_j \right\}, \quad  j=1,2,3.$$
The element of $\mathcal{S}^j_2$, say $x^j_2$, can be written by
\begin{align*}
x^1_2=\frac {p}{q}+\frac {s-1}{q^2}, \ \ \ \ \ \  x^2_2=\frac {p}{q}+\frac {1-s}{q^2}, \ \ \ \ \ \  x^3_2=\frac {p}{q}+\frac {t}{q^2},
\end{align*}  where $s\in \mathcal{S}, t\in \mathcal{Q}$.
By \eqref{eq2.7}, it is easy to see  that  the maximum and minimum can not be taken from $\mathcal{S}^1_2$ or $\mathcal{S}^2_2$. Noting that $\mathcal{Q}$ is bounded, we can do the same decomposition $k$ times for $\mathcal{S}^3_2$ and get the element $\frac {p}{q}+\frac {t}{q^{k+1}}$, which is between $\frac {p-1}{q-1}$ and $\frac {p+1}{q+1}$.  Hence the maximum and minimum can not be taken from $\mathcal{S}_2$ as well.

Denote by $M_1^*:=\sup\{x:x\in \mathcal{S}\}, m_1^*:=\inf\{x:x\in \mathcal{S}\}$.  Trivially $M_1^*=m_1^*=M_1=m_1=1$ when $p=q$. Otherwise, for the case $q>0$.

(a) If $p>q$, then $\frac {p-1}{q-1}>\frac {pq+q-2p}{q(q-1)}\ge 1$, and $M_1^*>1$.  If $|1-m_1^*|\leq M_1^*-1$, then $\frac {p}{q}+\frac {1-M_1^*}{q}\leq x \leq \frac {p}{q}+\frac {M_1^*-1}{q}$ for any $x\in\mathcal{S}$ by \eqref{eq2.7}. Then
\begin{align*}
M_1^*\leq\frac {p}{q}+\frac {M_1^*-1}{q}, \qquad  m_1^*\geq \frac {p}{q}+\frac {1-M^*_1}{q}.
\end{align*}
Hence $M^*_1\leq\frac{p-1}{q-1}, \    m_1^*\geq\frac {pq+q-2p}{q(q-1)}$.  If $|1-m_1^*|> M_1^*-1$, then $m_1^*<1$ and $\frac {p}{q}+\frac {m_1^*-1}{q}\leq x \leq \frac {p}{q}+\frac {1-m_1^*}{q}$ for any $x\in\mathcal{S}$ by \eqref{eq2.7}.
Then
\begin{align*}
m_1^*\geq\frac {p}{q}+\frac {m_1^*-1}{q},\qquad M_1^*\leq\frac {p}{q}+\frac {1-m_1^*}{q}.
\end{align*}
Hence $m_1^*\geq\frac{p-1}{q-1}, \  M_1^*\leq\frac {pq+q-2p}{q(q-1)}$, which is impossible.  Therefore,  $M^*_1=M_1=\frac{p-1}{q-1}, \ m_1^*=m_1=\frac {pq+q-2p}{q(q-1)}$.

(b) If $p<q$, then $\frac {p-1}{q-1}<\frac {pq+q-2p}{q(q-1)} \le 1$ and  $m_1^*< 1$. If $1-m_1^*\leq |M_1^*-1|$,  then $M_1^*>1$ and $\frac {p}{q}+\frac {1-M_1^*}{q}\leq x \leq \frac {p}{q}+\frac {M_1^*-1}{q}$ for any $x\in\mathcal{S}$ by \eqref{eq2.7}. Hence
\begin{align*}
M_1^*\leq\frac {p}{q}+\frac {M_1^*-1}{q},\qquad m_1^*\geq \frac {p}{q}+\frac {1-M^*_1}{q}.
\end{align*}
It follows that $M^*_1\leq\frac{p-1}{q-1}, \ m_1^*\geq\frac {pq+q-2p}{q(q-1)}$, which is also impossible.

If $1-m_1^*>| M_1^*-1|$, then $\frac {p}{q}+\frac {m_1^*-1}{q}\leq x \leq \frac {p}{q}+\frac {1-m_1^*}{q}$ for any $x\in\mathcal{S}$ by \eqref{eq2.7}.
Then
\begin{align*}
m_1^*\geq\frac {p}{q}+\frac {m_1^*-1}{q},\qquad M_1^*\leq\frac {p}{q}+\frac {1-m_1^*}{q},
\end{align*}
and $m_1^*\geq\frac{p-1}{q-1},\ M_1^*\leq\frac {pq+q-2p}{q(q-1)}$. Therefore, $m^*_1=m_1=\frac{p-1}{q-1}, \  M_1^*=M_1=\frac {pq+q-2p}{q(q-1)}$.

In conclusion, we finish the proof of (i). The similar argument can be used to prove (ii).
\end{proof}

\begin{pro}\label{th(2.2.1)}
 Let $p\geq 2, m \in \mathbb{Z}$, and $p+1<m<2p-1 $, let $M'_1=\max \{x:x\in \mathcal{S}'\}, m'_1=\min\{x:x\in \mathcal{S}'\}$, then
$M'_1=\frac{p}{p-1}, m'_1=\frac{p-2}{p-1}$.
\end{pro}

\begin{proof}
Let $\mathcal{S}'_i= \left \{\sum_{k \ge 1}ka_k p^{-k}: (a_1,a_2,\dots) \in \mathcal{A}_i \right\}, i=1,2,3.$  Then $\mathcal{S}'=\bigcup_{i=1}^3\mathcal{S}'_i$.  The element of $\mathcal{S}'_i$, say $x'_i$, can be written explicitly as follows:
\begin{align*}
x'_1=1+\frac{s'}{p}, \ \ \ \ \ \  x'_2=1+\frac {t'}{p},  \ \ \ \ \ \  x_3=1-\frac {s'}{p},
\end{align*}  where $s'\in \mathcal{S}', t'\in \mathcal{Q}'$.

Since $(p-1,p-1,p-1,\dots), (p+1,-p+1,-p+1,\dots)$ belong to $\mathcal{A}$, we have
\begin{equation*}
 \frac {p}{p-1},\quad \frac {p-2}{p-1}\in\mathcal{S}'.
\end{equation*}

Using the similar proof of  Proposition \ref{th(2.2)}, we conclude that $M'_1=\frac{p}{p-1}, m'_1=\frac{p-2}{p-1}$.
\end{proof}

\begin{pro}\label{th(2.3)}
Under the same assumption of Proposition \ref{th(2.2)}, and let $M_2 =\max \{x:x\in \mathcal{Q}\}, m_2 = \min\{x:x\in \mathcal{Q}\}$,
$M'_2 =\max \{x:x\in \mathcal{Q}'\}, m'_2 = \min\{x:x\in \mathcal{Q}'\}$.  Then
$M_2=| m_2 |=\frac {|p-q|}{|q|(|q|-1)}$, $M'_2=| m'_2 |=\frac {1}{p-1}$.
\end{pro}

\begin{proof}
By the symmetry of ${\mathcal B}$, we only need to calculate the maximum as $M_2=-m_2$, $M'_2=-m'_2$. Moreover,  note that $\mathcal{B}=\mathcal{B}_1\cup\mathcal{B}_2\cup\mathcal{B}_3$, it suffices to consider $\mathcal{B}_1$.

By making use of Proposition \ref{th(2.2)}, if $q>0$, when $p\geq q$, we have $M_1\geq m_1\ge 1$, so
\begin{align*}
\max\{|\sum_{k \ge 1} a_k q^{-k}|: \  (a_1,a_2,\dots) \in \mathcal{B}_1\}=\frac{M_1-1}{q}=\frac{p-q}{q(q-1)}.
\end{align*}
When  $p< q$, we have $1\ge M_1>m_1>0$, so
\begin{align*}
\max\{|\sum_{k \ge 1} a_k q^{-k}|: \  (a_1,a_2,\dots) \in \mathcal{B}_1\}=\frac{1-m_1}{q}=\frac{q-p}{q(q-1)}.
\end{align*}

If $q<0$, then $m_1<0, M_1< 1$, hence
\begin{align*}
&\max\{|\sum_{k \ge 1} a_k q^{-k}|: \  (a_1,a_2,\dots) \in \mathcal{B}_1\}=\frac{1-m_1}{q}=\frac{p-q}{q(q+1)},
\end{align*} and  $M_2=| m_2 |=\frac {|p-q|}{|q|(|q|-1)}$.

It follows from Proposition \ref{th(2.2.1)} that
\begin{align*}
\max\{|\sum_{k \ge 1} ka_k p^{-k}|: \  (a_1,a_2,\dots) \in \mathcal{B}_1\}=\frac{M'_1}{p}=\frac{1}{p-1}.
\end{align*}
Therefore, $M'_2=| m'_2 |=\frac {1}{p-1}$.
\end{proof}

\section{\bf Connectedness}

For the affine pair $(A,\mathcal {D})$ as in the last section, we let
\begin{eqnarray}\label{eq(3.1)}
  S_{i,j}\dbinom{x}{y}=A^{-1}\Bigg[\dbinom{x}{y}+\dbinom{i}{j}\Bigg],
\end{eqnarray}
where $i\in E_m, j\in E_n$. Then $\{S_{i,j}\}$ is the iterated function system that generates the self-affine set $T(A,\mathcal{D})$. According to (\ref{eq2.4}) and (\ref{eq(3.1)}),  $S_{i,j}(T)$ can be written as the form:
\begin{eqnarray}\label{eq(3.2)}
S_{i,j}(T)= \left\{\dbinom{\frac{i+p({\bf i})}{p}} {ar(i{\bf i})+\frac{j+q({\bf j})}{q}}: {\bf i}\in {\mathcal I}_1, {\bf j}\in {\mathcal I}_2\right\},
\end{eqnarray}
where $r(i{\bf i})={i}/{pq}+\sum_{k=1}^{\infty}r_{k+1} i_k$ for ${\bf i}=i_1i_2\cdots.$
\begin{lemma}\label{th(3.1)}
If $p+1<m<2p-1$, then $S_{i_1,j_1 }(T)\cap S_{i_2,j_2 }(T)\neq \emptyset  $ implies $\left| i_1-i_2 \right|\le 1$.
\end{lemma}

\begin{proof}
By the above, we assume $$S_{i_1,j_1}(T)= \left\{\dbinom{\frac{i_1+p({\bf i})}{p}} {ar(i_1{\bf i})+\frac{j_1+q({\bf j})}{q}}: {\bf i}\in {\mathcal I}_1, {\bf j}\in {\mathcal I}_2\right\}$$ and $$S_{i_2,j_2}(T)= \left\{\dbinom{\frac{i_2+p({\bf i}')}{p}} {ar(i_2{\bf i}')+\frac{j_2+q({\bf j}')}{q}}: {\bf i}'\in {\mathcal I}_1, {\bf j}'\in{\mathcal I}_2\right\}.$$

The nonempty intersection implies that
\begin{align*}
i_1-i_2\in\{p({\bf i})-p({\bf i}'): {\bf i},{\bf i}' \in {\mathcal I}_1\}=\{\sum_{k\ge 1} a_kp^{-k}: a_k\in \Delta E_m\}
\end{align*} where the last term is an interval $[-\frac{m-1}{p-1}, \frac{m-1}{p-1}]$. From the condition $p+1<m<2p-1$, we conclude that  $|i_1-i_2|=0$ or $1$.
\end{proof}

\begin{lemma}\label{th(3.2)}
Let $G_i(T)=\cup_{j=0}^{n-1}S_{i,j}(T)$ for $i=0,1,\dots, m-1$. If $p+1<m<2p-1$, and $n\geq (|q|+1)/2$. Then  $G_i(T)\cap G_{i+1}(T) \neq \emptyset$ if and only if  $|a|(|q|-2) \leq q^2(n-1)$.
 \end{lemma}

\begin{proof} $G_i(T)\cap G_{i+1}(T)\neq \emptyset  $ holds if  and  only  if there exist some $j,k$ such that $S_{i,j}(T)\cap S_{i+1,k}(T)\neq \emptyset$, that is, by (\ref{eq(3.2)}), there exists a sequence $\{a_k\} \in {\mathcal A}$ such that
\begin{align}\label{eq0}
{\frac{a}{pq}-a \sum_{k\ge 1} r_{k+1}a_k} \in \left\{\sum_{k\ge 0}j_kq^{-k}: j_k\in \Delta E_n \right\},
\end{align}
in which the last set is equal to $\left [-\frac{{n-1}}{\left | q \right |-1},\frac{{n-1}}{\left | q \right |-1}\right]$, as $n\ge (|q|+1)/2$.

It follows from \eqref{eq2.3} that: if $p=q$, then the left term of \eqref{eq0} becomes   $$\frac{a}{q^2}\sum_{k \ge 1} k a_kq^{-k}.$$
By Proposition \ref{th(2.2.1)}, \eqref{eq0} holds if and only if
$$
\min_{\{a_k\} \in {\mathcal A}}|\frac{a}{q^2}\sum_{k \ge 1} k a_kq^{-k}|=\left|\frac{am'_1}{q^2}\right|=\left|\frac{a(q-2)}{q^2(q-1)}\right|\le \frac{{n-1}}{\left | q \right |-1}.
$$

If $p\neq q$, then the left term of \eqref{eq0} becomes $$\frac{a}{q(q-p)}(\sum_{k \ge 1} a_kq^{-k}-1).$$  By Proposition \ref{th(2.2)},  $M_1>m_1\geq 1$ or $m_1<M_1\leq 1$ for any cases. Hence \eqref{eq0} is equivalent to
$$
\left| \frac{a}{q(q-p)}(M_1-1) \right|\leq \frac{{n-1}}{|q|-1} \quad\text{or}\quad  \left| \frac{a}{q(q-p)}(m_1-1)\right|\leq \frac{{n-1}}{|q|-1}.
$$
By substituting the values of $M_1, m_1$ in (\ref{eq2.5}), (\ref{eq2.6}), we conclude that \eqref{eq0} holds if and only if  $|a|(|q|-2)\leq q^2(n-1).$
\end{proof}

\begin{lemma}\label{th(3.4)}
Under the above assumption of Lemma \ref{th(3.2)}.  If $|q|>2$, then both $G_i(T)\cap G_{i+1}(T) \neq \emptyset$ and $S_{i,j}(T)\cap S_{i,j+1}(T)\neq \emptyset$ hold for any $i,j$ if and only if  $|q|(|q|-n)\leq|a|\leq \frac{q^2(n-1)}{|q|-2}$; if $|q|=2$, then $G_i(T)\cap G_{i+1}(T) \neq \emptyset$ and $S_{i,j}(T)\cap S_{i,j+1}(T)\neq \emptyset $ always hold.
\end{lemma}

\begin{proof}
If $|q|>2$, then by Lemma \ref{th(3.2)}, $G_i(T)\cap G_{i+1}(T) \neq \emptyset$ implies that $|a|\leq \frac{q^2(n-1)}{|q|-2}$.

By (\ref{eq(3.2)}), $S_{i,j}(T)\cap S_{i,j+1}(T)\neq \emptyset$ holds if and only if there exists a sequence $\{a_k\} \in {\mathcal B}$  such that
\begin{align*}
1+qa\sum_{k \ge 1} {r_{k+1}a_k} \in \left\{ \sum_{k \ge 1} j_kq^{-k}: j_k\in \Delta E_n \right\}= \left[ -\frac{{n-1}}{\left | q \right |-1},\frac{{n-1}}{\left | q \right |-1}\right].
\end{align*}
Equivalently,
\begin{align}\label{eq(3.8)}
  {qa\sum_{k \ge 1} {r_{k+1}a_k}}  \in  \left [ -\frac{{\left | q \right |}+n-2}{\left | q \right |-1}, -\frac{{\left | q \right |}-n}{\left | q \right |-1}\right].
\end{align}
If $n\geq |q|$, we can choose $a_k=0$ for all $k$ such that \eqref{eq(3.8)} holds for any $a$.  Hence for $|q|=2$, as $n\geq (|q|+1)/2=3/2$, then
$n\geq 2=|q|$ and $S_{i,j}(T)\cap S_{i,j+1}(T)\neq \emptyset $ always holds.

By Proposition \ref{th(2.3)}, we have
\begin{equation*}
\max_{\{a_k\} \in {\mathcal B}} |{qa\sum_{k \ge 1} {r_{k+1}a_k}}|=\left\{
\begin{array}{ll}
{\max}_{\{a_k\} \in {\mathcal B}}|\frac{a}{q}\sum_{k \ge 1} k a_kq^{-k}|=|\frac{a}{q(|q|-1)}| \quad & p=q\\
\\
{\max}_{\{a_k\} \in {\mathcal B}}|\frac {a}{q-p}\sum_{k \ge 1} a_kq^{-k}|=|\frac{a}{q(|q|-1)}| \quad & p\neq q.
\end{array}
\right.
\end{equation*}

If $n<|q|$, then \eqref{eq(3.8)} implies that
\begin{align*}
\frac{|q|-n}{|q|-1} \le \max_{\{a_k\} \in {\mathcal B}}|{qa\sum_{k \ge 1} {r_{k+1}a_k}}|=\left|\frac{a}{q(|q|-1)} \right|.
\end{align*}
Therefore, $|q|(|q|-n)\leq|a|\leq \frac{q^2(n-1)}{|q|-2}$.

For the sufficiency, if $|q|(|q|-n)\leq|a|\leq \frac{q^2(n-1)}{|q|-2}$ holds, then  $G_i(T)\cap G_{i+1}(T) \neq \emptyset$ by Lemma \ref{th(3.2)}, and
\begin{align*}
\frac{{\left | q \right |}-n}{\left | q \right |-1} \le \left|\frac{a}{q(|q|-1)} \right|= \max_{\{a_k\} \in {\mathcal B}}|{qa\sum_{k \ge 1} {r_{k+1}a_k}}|\leq\frac{|q|(n-1)}{(|q|-1)(|q|-2)}.
\end{align*}
As $|q|\geq 3$, if $n<|q|$ then
 \begin{align*}
 \frac{|q|(n-1)}{|q|-2}-(|q|+n-2)=\frac{-q^2+3|q|-4+2n}{|q|-2}\leq\frac{-q^2+5|q|-6}{|q|-2}=-(|q|-3)\leq 0.
\end{align*}
Hence $$\frac{|q|-n}{|q|-1}\leq  \max_{\{a_k\} \in {\mathcal B}} |{qa\sum_{k \ge 1} {r_{k+1}a_k}}|\leq\frac{|q|(n-1)}{(|q|-1)(|q|-2)}\leq\frac{{\left | q \right |}+n-2}{\left | q \right |-1},$$ which means that we can find the sequence $\{a_k\}_k\in {\mathcal B}$ satisfying \eqref{eq(3.8)}. If $n\ge |q|$, we choose $a_k=0$ for all $k$ such that \eqref{eq(3.8)}  holds.  Consequently, $S_{i,j}(T)\cap S_{i,j+1}(T)\neq \emptyset$ holds.
\end{proof}

\begin{lemma}\label{th(3.5)}
Under the above assumption of Lemma \ref{th(3.2)} and $|q|\ge 3$. Then both $S_{i,j}(T)\cap S_{i+1,j}(T)\neq \emptyset$ and $S_{i,j}(T)\cap S_{i+1,j+1}(T)\neq \emptyset$ hold, or  both $S_{i,j}(T)\cap S_{i+1,j}(T)\neq \emptyset$ and  $S_{i,j+1}(T)\cap S_{i+1,j}(T)\neq \emptyset$ hold for all  $i, j$ if and only if $|q|-n\leq|a|\leq \frac{|q|(n-1)}{|q|-2}$.
 \end{lemma}

\begin{proof}
Analogous to the above proofs, $S_{i,j}(T)\cap S_{{i+1},j}(T)\neq \emptyset$ holds if and only if there exists a sequence $\{a_k\} \in {\mathcal A}$ such that
 \begin{align}\label{eq(3.5.2)}
 {\frac{a}{p}-qa \sum_{k \ge 1} r_{k+1}a_k} \in \left\{ \sum_{k \ge 1} j_kq^{-k} , j_k\in \Delta E_n \right\}=  \left[ -\frac{n-1}{\left | q \right |-1},\frac{n-1}{\left | q \right |-1}\right].
\end{align}
By \eqref{eq2.3}, \eqref{eq(3.5.2)} holds if and only if
\begin{equation*}
\min_{\{a_k\} \in {\mathcal A}}|{\frac{a}{p}-qa \sum_{k \ge 1} r_{k+1}a_k}|=\left\{
\begin{array}{ll}
\min_{\{a_k\} \in {\mathcal A}}|\frac{a}{q}\sum_{k \ge 1} k a_kq^{-k}| \le \frac{{n-1}}{|q|-1} \quad & p=q\\
\\
\min_{\{a_k\} \in {\mathcal A}}|\frac{a}{(q-p)}(\sum_{k \ge 1} a_kq^{-k}-1)| \le \frac{{n-1}}{|q|-1} \quad & p\neq q.
\end{array}
\right.
\end{equation*}
Using the same discussion as in Lemma \ref{th(3.2)}, we obtain that $S_{i,j}(T)\cap S_{{i+1},j}(T)\neq \emptyset$  if and only if $|a|\leq \frac{|q|(n-1)}{|q|-2}.$

Similarly, if $S_{i,j}(T)\cap S_{{i+1},{j+1}}(T)\neq \emptyset$, then there exists a sequence  $\{a_k\} \in {\mathcal A}$ such that
\begin{align*}
{\frac{a}{pq}+\frac{1}{q}-a \sum_{k \ge 1} r_{k+1}a_k} \in \left\{ \sum_{k \ge 1} j_k q^{-k-1}:  j_k\in \Delta E_n \right\}=\frac {1}{\left| q\right|  } \left[ -\frac{n-1}{\left | q \right |-1},\frac{n-1}{\left | q \right |-1}\right].
\end{align*}
Equivalently,
\begin{align}\label{eq(3.6.1)}
 \frac{a}{p}-aq \sum_{k \ge 1} r_{k+1}a_k  \in  \left [ -\frac{{\left | q \right |}+n-2}{\left | q \right |-1}, -\frac{{\left | q \right |}-n}{\left | q \right |-1}\right].
\end{align}
If $S_{i,j+1}(T)\cap S_{{i+1},{j}}(T)\neq \emptyset$,  then there exists a sequence  $\{a_k\}\in {\mathcal A}$ such that
\begin{align*}
{\frac{a}{pq}+\frac{-1}{q}-a \sum_{k \ge 1} r_{k+1}a_k} \in \left\{ \sum_{k \ge 1} j_kq^{-k-1}: j_k\in \Delta E_n \right\} =\frac {1}{\left| q\right|  } \left[ -\frac{n-1}{\left | q \right |-1},\frac{n-1}{\left | q \right |-1}\right].
\end{align*}
Equivalently,
\begin{align}\label{eq(3.7.1)}
 \frac{a}{p}-aq \sum_{k \ge 1} r_{k+1}a_k  \in  \left [ \frac{{\left | q \right |}-n}{\left | q \right |-1}, \frac{{\left | q \right |}+n-2}{\left | q \right |-1}\right].
 \end{align}
 By  Propositions \ref{th(2.2)} and \ref{th(2.2.1)}, we have
\begin{align*}
  \max_{\{a_k\} \in {\mathcal A}}|\frac{a}{p}-qa \sum_{k \ge 1} r_{k+1}a_k|=\frac{|a|}{|q|-1}.
\end{align*}
If $n\geq |q|$, then $|q|-n\leq|a|$ obviously. If $n<|q|$, then \eqref{eq(3.6.1)} and \eqref{eq(3.7.1)} imply that
\begin{align*}
  \frac{|a|}{|q|-1}=\max_{\{a_k\} \in {\mathcal A}}|\frac{a}{p}-qa \sum_{k \ge 1} r_{k+1}a_k|\geq\frac{|q|-n}{|q|-1}.
\end{align*}
Hence  $|a|\ge |q|-n$. The sufficiency is similar to Lemma \ref{th(3.4)}, we omit the details.
\end{proof}

Now by making use of the above four lemmas, we can prove our main theorem.

\begin{thm}\label{th(3.6)}
 Let $p,q,\in\mathbb{Z}$ with  $|p|, |q|\geq 2$, $a\in \mathbb{R}$, and let  $|p|+1<m<2|p|-1 $,
$n\geq (|q|+1)/2$. Let
  \begin{equation*}
 A=\begin{pmatrix}
 p& 0 \\
 -a&q
\end{pmatrix} , \quad  \mathcal {D}= E_m\times E_n.
\end{equation*}

(i) If $|q|=2$, then $T(A,\mathcal{D})$ is connected for any $a$.

(ii) If $|q|\ge 3$, then $T(A,\mathcal{D})$ is connected if and only if
\begin{align}\label{eq(3.18.1)}
|q|(|q|-n)\leq |a| \leq \frac{q^2(n-1)}{|q|-2},
\end{align}
or
\begin{align}\label{eq(3.18.2)}
|q|-n\leq |a| \leq \frac{|q|(n-1)}{|q|-2}.
\end{align}
\end{thm}

\begin{proof}
(i) is trivial by Lemma \ref{th(3.4)}. It suffices to show  (ii). If \eqref{eq(3.18.1)} holds, then
\begin{eqnarray}\label{eq(3.20)}
S_{i,j}(T)\cap S_{i,j+1}(T)\neq \emptyset \quad \text{for any} \quad i,j,
 \end{eqnarray}
and $G_i(T)\cap G_{i+1}(T)\neq \emptyset$  by Lemma \ref{th(3.4)}. Since $G_i(T)=\cup_{j=0}^{n-1} S_{i,j}(T)$, for each $0\leq i\leq m-1$, there exist $0\leq j_i, k_i \leq n-1$ such that
\begin{eqnarray}\label{eq(3.21)}
S_{i,j_i}(T)\cap S_{i+1,k_i}(T)\neq \emptyset.
 \end{eqnarray}
We use \eqref{eq(3.20)}, \eqref{eq(3.21)} to select a sequence $\{\psi_i\}_{i=1}^N$ from $\{S_{i,j}\}$ in the following order:

$S_{0,0}, S_{0,1},\dots,S_{0,n-2}, S_{0,n-1},S_{0,n-2},\dots, S_{0,j_0}, S_{1,k_0},\dots, S_{1,n-1}, S_{1,n-2},\dots, S_{1,0}, \\
S_{1,1},\dots, S_{1,j_1}, S_{2,k_1},\dots,S_{2,n-1}, S_{2,n-2},\dots,S_{2,0}, S_{2,1},\dots, S_{2,j_2},\dots, S_{m-1,n-1}.$

\noindent Then each $S_{i,j}$ appears at least once in the sequence $\{\psi_i\}_{i=1}^N$ and
$$
\psi_i(T)\cap\psi_{i+1}(T)\neq\emptyset, \qquad  1\leq j \leq N.
$$
That implies that $T$ is connected by Lemma  \ref{th(2.1)}.

If \eqref{eq(3.18.2)} holds, by Lemma \ref{th(3.5)}, then both $S_{i,j}(T)\cap S_{i+1,j}(T)\neq \emptyset$ and  $S_{i,j}(T)\cap S_{i+1,j+1}(T)\neq \emptyset$ hold, or both $S_{i,j}(T)\cap S_{i+1,j}(T)\neq \emptyset$ and  $S_{i,j+1}(T)\cap S_{i+1,j}(T)\neq \emptyset$ hold.

For the first case that $S_{i,j}(T)\cap S_{i+1,j}(T)\neq \emptyset, S_{i,j}(T)\cap S_{i+1,j+1}(T)\neq \emptyset$, we can select a sequence $\{\psi_i\}_{i=1}^N$ from $\{S_{i,j}\}$ in the following order:

\noindent (a) when $m$ is even

$S_{0,n-1},S_{1,n-1},S_{0,n-2},S_{1,n-2},\dots,S_{0,0}, S_{1,0},S_{2,0},S_{1,0},S_{2,1},S_{1,1}, S_{2,2},\cdots
S_{1,n-2},S_{2,n-1}$\\
$\cdots ,S_{m-2,n-1},S_{m-1,n-1},S_{m-2,n-2},S_{m-1,n-1}, S_{m-2,n-3},\cdots S_{m-2,0},S_{m-1,0}.$

\noindent (b) when $m$ is odd

$S_{0,n-1},S_{1,n-1},S_{0,n-2},S_{1,n-2},\dots,S_{0,0}, S_{1,0},S_{2,0},S_{1,0},S_{2,1},S_{1,1}, S_{2,2},\cdots
S_{1,n-2},S_{2,n-1}$\\
$\cdots ,S_{m-2,0}, S_{m-1,0},S_{m-2,0},S_{m-1,1},S_{m-2,1},S_{m-1,2}, S_{m-2,2},\cdots S_{m-2,n-2},S_{m-1,n-1}.$

\noindent Each $S_{i,j}$ appears at least once in the sequence $\{\psi_i\}_{i=1}^N$ and
$$
\psi_i(T)\cap\psi_{i+1}(T)\neq\emptyset, \qquad  1\leq j \leq N.
$$ Hence $T$ is connected by Lemma \ref{th(2.1)}.

Using the similar way, we can handle the other case that $S_{i,j}(T)\cap S_{i+1,j}(T)\neq \emptyset $ and  $S_{i,j+1}(T)\cap S_{i+1,j}(T)\neq \emptyset$.

For the necessary part, we prove it by contrapositive.

(a)  If $|a|>\frac{q^2(n-1)}{|q|-2}$, then $G_0(T)\cap G_i(T)=\emptyset$ for $1\leq i \leq m-1$ by Lemmas \ref{th(3.1)} and  \ref{th(3.2)},  implying $T$ is disconnected.

(b) If $0 \leq|a|< |q|-n$,  we show $S_{i,t}(T)\cap S_{i,l}(T)=\emptyset$ for any $|l-t|\geq 1$. If otherwise, $S_{i,t}(T)\cap S_{i,l}(T)\neq \emptyset$ holds for some $t, l$ with $|l-t|\geq 1$, as \eqref{eq(3.8)}, then there exists a sequence $\{a_k\} \in {\mathcal B}$  such that
$${qa\sum_{k \ge 1} {r_{k+1}a_k}}  \in  \left [ -\frac{(|q|-1)(l-t)+(n-1)}{|q|-1}, -\frac{(|q|-1)(l-t)-(n-1)}{\left | q \right |-1}\right].$$
That implies
\begin{align}\label{eq(3.14)}
 \left|\frac{a}{q(|q|-1)} \right|=\max_{\{a_k\} \in {\mathcal B}}|{qa\sum_{k \ge 1} {r_{k+1}a_k}}|\geq\frac{(|q|-1)|l-t|-(n-1)}{\left | q \right |-1},
\end{align} which is impossible as $0\leq a<|q|-n$.  Hence  $S_{i,n-1}(T)\cap S_{i,j}(T)=\emptyset$ for $0\leq j\leq n-2$.

Similarly we can show  $S_{i,u}(T)\cap S_{i+1,v}(T)= \emptyset$  for any $|v-u|\geq 1$. If otherwise,  $S_{i,u}(T)\cap S_{i+1,v}(T)\ne \emptyset$  for some $|v-u|\geq 1$, then there exists a sequence $\{a_k\} \in {\mathcal A}$  such that
\begin{align}\label{eq(3.15)}
 \frac{a}{p}-aq \sum_{k \ge 1} r_{k+1}a_k  \in  \left [ -\frac{(|q|-1)(v-u)+(n-1)}{|q|-1}, -\frac{(|q|-1)(v-u)-(n-1)}{\left | q \right |-1}\right].
\end{align} That implies
\begin{align*}
 \frac{|a|}{|q|-1} =\max_{\{a_k\} \in {\mathcal A}}|\frac{a}{p}-aq \sum_{k \ge 1} r_{k+1}a_k |\geq\frac{(|q|-1)|v-u|-(n-1)}{\left | q \right |-1},
\end{align*} which is also impossible as $0\leq a<|q|-n$. Hence we have  $S_{i,n-1}(T)\cap S_{i+1,j}(T)=\emptyset$ for $0\leq j\leq n-2$.

Let $\Omega_1=\bigcup_{i=0}^{m-1}S_{i,n-1}(T), \Omega_2=\bigcup_{i=0}^{m-1}\bigcup_{j=0}^{n-2}S_{i,j}(T)$, the above arguments  and Lemma \ref{th(3.1)} yield that $\Omega_1\cap\Omega_2=\emptyset$, proving that $T$ is  disconnected.

(c) If $\frac{|q|(n-1)}{|q|-2}<|a|<|q|(|q|-n)$, then we have $|a|<|q|(|q|-n)$ and \eqref{eq(3.14)} imply that $S_{i,t}(T)\cap S_{i,l}(T)=\emptyset$ for $|l-t|\geq 1$; and $\frac{|q|(n-1)}{|q|-2}<|a|$ implies $S_{i,j}(T)\cap S_{i+1,j}(T)=\emptyset$ by using the first part of the proof of Lemma \ref{th(3.5)}. Moreover, if $aq>0$,  Propositions \ref{th(2.2)} and \ref{th(2.2.1)} imply that
\begin{align*}
-\frac{|a|}{|q|-1}\leq\frac{a}{p}-qa \sum_{k \ge 1} r_{k+1}a_k\leq -\frac{|a|(|q|-2)}{|q|(|q|-1)}<0
\end{align*}
for all $\{a_k\} \in {\mathcal A}$.  Then \eqref{eq(3.15)} does not hold if $v-u\leq -1$ as the interval of \eqref{eq(3.15)} lies in ${\mathbb R}^+$. Hence, $S_{0,n-1}(T)\cap S_{1,j}(T)=\emptyset$ for $0\leq j\leq n-2$. By the above discussion and  Lemmas \ref{th(3.1)}, it concludes that $S_{0,n-1}(T)\cap S_{i,j}(T)=\emptyset$ for all $i,j$ with $(i,j)\neq(0,n-1)$.  If $aq<0$, by using the similar argument we can prove $S_{0,0}(T)\cap S_{i,j}(T)=\emptyset$ for  all $i,j$ with $(i,j)\neq(0,0)$. Therefore, $T$ is disconnected.
\end{proof}

\section{\bf Tile digit set}

For the affine pair $(A, {\mathcal D})$ as in (\ref{eq 1.2}). Let $D_{A,k}:=\{\sum_{i=0}^{k-1}A^id_{j_i} : d_{j_i}\in\mathcal {D}\}$ for $k\ge 1$ and $D_{A,\infty}=\bigcup_{k=1}^\infty D_{A,k}$. The following are well-known equivalent conditions for the self-affine set $T(A, {\mathcal D})$ to be a tile due to  Lagarias and Wang.

\begin{lemma}[\cite{LW1}]\label{th4.1}
That $T:=T(A,\mathcal {D})$ is a self-affine tile is equivalent to either one of the following conditions:

(i)\ $\mu(T)>0$, where $\mu$ is the lebesgue measure;

(ii) \ $\overline{T^\circ}=T$ and $\mu(\partial T)=0$;

(iii) \ $\# D_{A,k}=|\det(A)|^{k}$ for all $k\geq 1$ and $D_{A,\infty}$ is a uniformly discrete set, i.e., there exists $\delta>0$ such that $\|u-v\|>\delta$ for any distinct $u,v \in D_{A,\infty}$.
\end{lemma}

Let $
A=\begin{pmatrix}
 p& 0 \\
 -a&q
\end{pmatrix},
$
then
\begin{equation}\label {eq4.1}
A^{k}=\begin{pmatrix}
 p^{k} &0\\
R_ka&q^{k}
\end{pmatrix},\quad k\geq 1
\end{equation}
where
\begin{equation*}
R_k:=\left\{
\begin{array}{ll}
(p^{k}-q^{k})/(q-p)& p\neq q \\
-kq^{k-1} & p=q.
\end{array}
\right.
\end{equation*}

\begin{pro}\label{th(4.4)}
Let $|p|, |q|\geq 2$ be integers and $a\in \mathbb{R}$, let $E_{m}=\{0,1,\dots, m-1\}$, $E_{n}=\{0,1,\dots, n-1\}$ with $|pq|=mn$.  Let
\begin{equation*}
 A=\begin{pmatrix}
 p& 0 \\
 a&q
\end{pmatrix}  \quad \text{and}\quad  \mathcal{D}=E_{m}\times E_{n}.
\end{equation*} Then we have

(i)   $m<|p|$,  $T(A,\mathcal {D})$ is not a tile.

(ii)  $m=|p|$, $T(A,\mathcal {D})$ is  a tile.

(iii)  $m> |p|$,

  \ \ \ (a) If $a=0$, then  $T(A,\mathcal {D})$ is not  a tile;

  \ \ \ (b) If $a\in {\mathbb Z}\setminus \{0\}$  and $2n-1\geq |q|$,  then $T(A,\mathcal {D})$ is not a tile.
\end{pro}

\begin{proof}
By the definition of $D_{A, k+1}$ (here $k\ge 0$) and (\ref{eq4.1}), it is easy to get that
\begin{eqnarray*}
D_{A,k+1}=\left\{\dbinom{\sum_{i=0}^k x_ip^i}
{\sum_{i=0}^k y_iq^i+a\sum_{i=1}^k R_ix_i}:x_i\in E_{m},y_i\in E_{n}\right\}.
\end{eqnarray*}
Then $\# D_{A,k+1}=|pq|^{k+1}$ if and only if
\begin{eqnarray}\label{eq(4.4)}
\dbinom{\sum_{i=0}^k x_ip^i}{\sum_{i=0}^k y_iq^i+a\sum_{i=1}^k R_ix_i}\neq \dbinom{\sum_{i=0}^k x'_ip^i}{\sum_{i=0}^k y'_iq^i+a\sum_{i=1}^k R_ix'_i}
\end{eqnarray}
for different sequences $\{x_0,x_1,\dots, x_k, y_0,y_1,\dots,y_k\}$ and $\{x'_0,x'_1,\dots, x'_k, y'_0,y'_1,\dots,y'_k\}$.

(i) If $m<|p|$, then $n>|q|$. We consider $D_{A,2}$ by letting $(x_0,x_1, y_0,y_1)=(0,0,q,0)$ and $(x'_0,x'_1,y'_0,y'_1)=(0,0,0,1)$, then
\begin{eqnarray*}
\dbinom{x_0+x_1p}{y_0+y_1q+R_1x_1}= \dbinom{x'_0+x'_1p}{y'_0+y'_1q+R_1x'_1}=\dbinom{0}{q}.
\end{eqnarray*} Hence $\# D_{A,2}<|pq|^{2}$ and $T(A,\mathcal {D})$ is not a tile by Lemma \ref{th4.1}.

(ii) If $m=|p|$, then $n=|q|$.  $T(A,\mathcal {D})$ is always a tile \cite{DL11}.

(iii) If $m>|p|$, then $n<|q|$. If $a=0$, then the set $\{y\in{\mathbb R}: (x,y)^t\in T \ \text{for some } x\in{\mathbb R}\}$ has one-dimensional Lebesgue measure zero, so $T$ can not be a tile. Hence (a) is true.

(b) Let $(x_0,x_1,\dots, x_k)=(p,0,\dots,0)$ and $(x'_0,x'_1,\dots, x'_k)=(0,1,0,\dots,0)$ where $k\geq 2$. Then
\begin{eqnarray*}
\left\{ \begin{aligned}
(x_0-x'_0)+(x_1-x'_1)p+\cdots+(x_k-x'_k)p^{k}&=0 \\
a(R_1(x_1-x'_1)+\cdots+R_k(x_k-x'_k))&=-a.
\end{aligned} \right.
\end{eqnarray*}
Since $2n-1\geq |q|$,  the set
\begin{eqnarray*}
\left\{(y_0-y'_0)+(y_1-y'_1)q+\cdots+(y_k-y'_k)q^k: \   y_i,y'_i\in E_{n} \right\}
\end{eqnarray*}
consists of  the integers in the interval $[-\frac{(n-1)(|q|^{k+1}-1)}{|q-1|},\frac{(n-1)(|q|^{k+1}-1)}{|q-1|}]$. Let $k_0$ is
the least number such that $\frac{(n-1)(|q|^{k_0+1}-1)}{|q-1|}\geq |a|$. Then there exist two sequences
$\{y_0,y_1,\dots, y_{k_0}\}$ and $\{y'_0,y'_1,\dots, y'_{k_0}\}$ such that
$$(y_0-y'_0)+(y_1-y'_1)q+\cdots+(y_{k_0}-y'_{k_0})q^{k_0}=a.$$
Thus we find two different sequences $\{x_0,x_1,\dots, x_{k_0}, y_0,y_1,\dots,y_{k_0}\}$ and \\
 $\{x'_0,x'_1,\dots, x'_{k_0}, y'_0,y'_1,\dots,y'_{k_0}\}$ such that (\ref{eq(4.4)}) does not hold for the $k_0$. Therefore, $\# D_{A,k_0+1}<|pq|^{k_0+1}$ and $T(A,\mathcal {D})$ is not a tile.
\end{proof}

\end{document}